\documentclass{amsart}
\usepackage{amssymb}
\usepackage{ifpdf}
\ifpdf
 \usepackage[hyperindex,pagebackref]{hyperref}%
\else
 \expandafter\ifx\csname dvipdfm\endcsname\relax
 \usepackage[hypertex,hyperindex,pagebackref]{hyperref}%
 \else
 \usepackage[dvipdfm,hyperindex,pagebackref]{hyperref}%
 \fi
\fi
\theoremstyle{plain}
\newtheorem{theorem}{Theorem}[section]
\newtheorem{lem}{Lemma}[section]
\theoremstyle{remark}
\newtheorem{rem}{Remark}[section]

\DeclareMathOperator{\arcsinh}{arcsinh}
\numberwithin{equation}{section}
\allowdisplaybreaks[4]

\begin{document}

\title[Sharp bounds for Neuman-S\'andor mean]
{Sharp bounds in terms of the power of the contra-harmonic mean for Neuman-S\'andor mean}

\author[W.-D. Jiang]{Wei-Dong Jiang}
\address[Jiang]{Department of Information Engineering, Weihai Vocational College, Weihai City, Shandong Province, 264210, China}
\email{\href{mailto: W.-D. Jiang <jackjwd@163.com>}{jackjwd@163.com}, \href{mailto: W.-D. Jiang <jackjwd@hotmail.com>}{jackjwd@hotmail.com}}

\author[F. Qi]{Feng Qi}
\address[Qi]{School of Mathematics and Informatics, Henan Polytechnic University, Jiaozuo City, Henan Province, 454010, China}
\email{\href{mailto: F. Qi <qifeng618@gmail.com>}{qifeng618@gmail.com}, \href{mailto: F. Qi <qifeng618@hotmail.com>}{qifeng618@hotmail.com}, \href{mailto: F. Qi <qifeng618@qq.com>}{qifeng618@qq.com}}
\urladdr{\url{http://qifeng618.wordpress.com}}



\begin{abstract}
In the paper, the authors obtain sharp bounds in terms of the power of the contra-harmonic mean for Neuman-S\'andor mean.
\end{abstract}

\thanks{This work was supported in part by the Project of Shandong Province Higher Educational Science and Technology Program under grant No. J11LA57.}

\thanks{This paper was typeset using \AmS-\LaTeX}

\maketitle

\section{Introduction}

For positive numbers $a,b>0$ with $a\neq b$, the second Seiffert mean $T(a,b)$, the root-mean-square $S(a,b)$, Neuman-S\'andor mean $M(a,b)$, and the contra-harmonic mean $C(a,b)$ are respectively defined in~\cite{Neuman-Sandor-Pannon, back10} by
\begin{align}\label{eq1.2}
T(a,b)&=\frac{a-b}{2\arctan[(a-b)/(a+b)]}, & S(a,b)&=\sqrt{\frac{a^2+b^2}{2}}\,,\\
\label{eq1.3}
M(a,b)&=\frac{a-b}{2\arcsinh[(a-b)/(a+b)]}, & C(a,b)&=\frac{a^2+b^2}{a+b}.
\end{align}
It is well known~\cite{Neuman-jmi-06-62, Neuman-jmi-601-609, Neuman-Sandor-Pannon2} that the inequalities
\begin{equation*}
M(a,b)<T(a,b)<S(a,b)<C(a,b)
\end{equation*}
hold for all $a,b>0$ with $a\neq b$ .
\par
In~\cite{seiffert15, seiffert10}, the inequalities
\begin{equation}\label{eq1.5}
S(\alpha a+(1-\alpha)b,\alpha b+(1-\alpha)a)<T(a,b)<S(\beta a+(1-\beta)b,\beta b+(1-\beta)a)
\end{equation}
and
\begin{equation}\label{eq1.6}
C(\lambda a+(1-\lambda)b,\lambda b+(1-\lambda)a)<T(a,b)<C(\mu a+(1-\mu)b,\mu b+(1-\mu)a)
\end{equation}
were proved to be valid for $\frac12<\alpha, \beta, \lambda, \mu<1$ and for all $a,b>0$ with $a\neq b$ if and only if
\begin{equation}
\begin{aligned}
\alpha & \le \frac12\biggl(1+\sqrt{\frac{16}{\pi^2}-1}\,\biggr), & \beta & \ge\frac{3+\sqrt{6}\,}6, \\
\lambda & \le \frac12\biggl(1+\sqrt{\frac4\pi-1}\,\biggr), & \mu &\ge\frac{3+\sqrt{3}\,}6
\end{aligned}
\end{equation}
respectively. In~\cite{Jiang-Neuman.tex}, the double inequality
\begin{equation}\label{eq1.7}
S(\alpha a+(1-\alpha)b,\alpha b+(1-\alpha)a)<M(a,b)<S(\beta
a+(1-\beta)b,\beta b+(1-\beta)a)
\end{equation}
was proved to be valid for $\frac12<\alpha,\beta<1$ and for all $a,b>0$ with $a\neq b$ if and only if
\begin{equation}
\alpha\le \frac{1}{2}\Biggl\{1+\sqrt{\frac{1}{\bigl[\ln \bigl(1+\sqrt{2}\,\bigr)\bigr]^2}-1}\,\Biggr\}\quad \text{and}\quad \beta\ge \frac{3+\sqrt{3}\,}{6}.
\end{equation}
For more information on this topic, please refer to recently published papers~\cite{back3, jiang-aadm4, Li-Long-Chu-JIA-12, background-Jiang-Qi-revised.tex} and references cited therein.
\par
For $t\in\bigl(\frac12,1\bigr)$ and $p\ge\frac12$, let
\begin{equation}\label{eq1.8}
Q_{t,p}(a,b)=C^{p}(ta+(1-t)b, tb+(1-t)a)A^{1-p}(a,b),
\end{equation}
where $A(a,b)=\frac{a+b}2$ is the classical arithmetic mean of $a$ and $b$. Then, by definitions in~\eqref{eq1.2} and~\eqref{eq1.3}, it is easy to see that
\begin{align*}
Q_{t,1/2}(a,b)&=S(ta+(1-t)b, tb+(1-t)a),\\
Q_{t,1}(a,b)&=C(ta+(1-t)b, tb+(1-t)a),
\end{align*}
and $Q_{t,p}(a,b)$ is strictly increasing with respect to $t\in\bigl(\frac12,1\bigr)$.
\par
Motivating by results mentioned above, we naturally ask a question: What are the greatest value $t_1=t_1(p)$ and the least value $t_{2}=t_{2}(p)$ in $\bigl(\frac12,1\bigr)$ such that the double inequality
\begin{equation}\label{eq1.9}
Q_{t_1,p}(a,b)<M(a,b)<Q_{t_{2},p}(a,b)
\end{equation}
holds for all $a,b>0$ with $a\neq b$ and for all $p\ge\frac12$?
\par
The aim of this paper is to answer this question. The solution to this question may be stated as the following Theorem~\ref{th1.1}.

\begin{theorem}\label{th1.1}
Let $t_1,t_{2}\in\bigl(\frac12,1\bigr)$ and $p\in\bigl[\frac12,\infty\bigr)$. Then the double inequality~\eqref{eq1.9} holds for all $a,b>0$ with $a\neq b$ if and only if
\begin{equation}
t_1\le\frac12\Biggl[1+\sqrt{\biggl(\frac1{t^*}\biggr)^{1/p}-1}\,\Biggr]\quad \text{and}\quad t_{2}\ge \frac12\biggl(1+\frac1{\sqrt{6p}\,}\biggr),
\end{equation}
where
\begin{equation}\label{t-star-dfn}
t^*=\ln (1+\sqrt{2}\,)=0.88\dotsc.
\end{equation}
\end{theorem}

\begin{rem}
When $p=\frac12$ in Theorem~\ref{th1.1}, the double inequality~\eqref{eq1.9} becomes~\eqref{eq1.7}.
\end{rem}

\begin{rem} If taking $p=1$ in Theorem~\ref{th1.1}, we can conclude that the double inequality
\begin{equation}\label{eq1.10}
C(\lambda a+(1-\lambda)b,\lambda b+(1-\lambda)a)<M(a,b)<C(\mu a+(1-\mu)b,\mu b+(1-\mu)a)
\end{equation}
holds for all $a,b>0$ with $a\neq b$ if and only if
\begin{equation}
\frac12<\lambda\le\frac12\Biggl[1+\sqrt{\frac{1}{\ln\bigl(1+\sqrt{2}\,\bigr)}-1}\,\Biggr]\quad \text{and}\quad 1>\mu\ge\frac12\biggl(1+\frac{\sqrt{6}\,}6\biggr).
\end{equation}
\end{rem}

\section{Lemmas}

In order to prove Theorem~\ref{th1.1}, we need the following lemmas.

\begin{lem}[{\cite[Theorem~1.25]{anderson1}}]\label{mon-quotient-lem}
For $-\infty<a<b<\infty$, let $f,g:[a,b]\to{\mathbb{R}}$ be continuous on $[a,b]$ and differentiable on $(a,b)$. If $g'(x)\neq 0$ and $\frac{f'(x)}{g'(x)}$ is strictly increasing \textup{(}or strictly decreasing, respectively\textup{)} on $(a,b)$, so are the functions
\begin{equation}
\frac{f(x)-f(a)}{g(x)-g(a)}\quad \text{and}\quad \frac{f(x)-f(b)}{g(x)-g(b)}.
\end{equation}
\end{lem}

\begin{lem}\label{g-3-lem}
The function
\begin{equation}\label{g-3-dfn}
h(x)=\frac{\bigl(1+x^2\bigr)\arcsinh {x}}x
\end{equation}
is strictly increasing and convex on $(0,\infty)$.
\end{lem}

\begin{proof}
This follows from the following arguments:
\begin{gather*}
h'(x)=\frac{x \sqrt{1+x^2}\,-\arcsinh x+x^2\arcsinh{x}}{x^2}
\triangleq\frac{h_1(x)}{x^2},\\
h_1'(x)=x \biggl(\frac{3 x}{\sqrt{1+x^2}}+2 \arcsinh{x}\biggr)
\triangleq xh_2(x),\\
h_2'(x)=\frac{5+2 x^2}{(1+x^2)^{3/2}}
>0
\end{gather*}
on $(0,\infty)$ and
\begin{equation*}
\lim_{x\to0^+}h_1(x)=\lim_{x\to0^+}h_2(x)=0.\qedhere
\end{equation*}
\end{proof}

\begin{lem}\label{f-u-p-x-lem}
For $u\in[0,1]$ and $p\ge\frac12$, let
\begin{equation}\label{eq2.1}
f_{u,p}(x)=p\ln(1+ux^2)-\ln{x}+\ln{\arcsinh{x}}
\end{equation}
on $(0,1)$. Then
the function $f_{u,p}(x)$ is positive if and only if $6pu\ge 1$ and it is negative if and only if $1+u\le\bigl(\frac1{t^*}\bigr)^{1/p}$, where $t^*$ is defined by~\eqref{t-star-dfn}.
\end{lem}

\begin{proof}
It is ready that
\begin{equation}\label{eq2.2}
\lim_{x\to 0^+}f_{u,p}(x)=0
\end{equation}
and
\begin{equation}\label{eq2.3}
\lim_{x\to1^-}f_{u,p}(x)=p\ln(1+u)+\ln(t^*).
\end{equation}
\par
An easy computation yields
\begin{equation}
\begin{aligned}\label{eq2.4}
f_{u,p}'(x)&=\frac{2pu x}{1+u x^2}+\frac{1}{\sqrt{1+x^2}\,\arcsinh{x}}-\frac{1}{x}\\
&=\frac{u\bigl[(2p-1)x^2\sqrt{1+x^2}\,\arcsinh{x}+x^3\bigr] -\bigl[\sqrt{1+x^2}\,\arcsinh{x}-x\bigr]}{x(1+ux^2)\sqrt{1+x^2}\,\arcsinh{x}}\\
&=\frac{(2p-1)x^2\sqrt{1+x^2}\,\arcsinh{x}+x^3}{x(1+ux^2)\sqrt{1+x^2}\,\arcsinh{x}} \biggl[u-\frac{g_1(x)}{g_2(x)}\biggr],
\end{aligned}
\end{equation}
where
\begin{equation*}
g_1(x)=\arcsinh{x}-\frac{x}{\sqrt{1+x^2}\,}\quad \text{and}\quad
g_{2}(x)=(2p-1)x^2\arcsinh{x}+\frac{x^3}{\sqrt{1+x^2}\,}.
\end{equation*}
Furthermore, we have
\begin{equation}\label{eq2.5}
g_1(0)=g_{2}(0)=0
\end{equation}
and
\begin{equation}\label{eq2.6}
\frac{g_1'(x)}{g_2'(x)}=\frac{1}{2(2p-1)\sqrt{1+x^2}\,h(x)+(2p+1)x^2+2p+2},
\end{equation}
where $h(x)$ is defined by~\eqref{g-3-dfn}. From Lemma~\ref{g-3-lem}, it follows that the quotient $\frac{g_1'(x)}{g_2'(x)}$ is strictly decreasing on $(0,1)$. Accordingly, from Lemma~\ref{mon-quotient-lem} and~\eqref{eq2.5}, it is deduced that the ratio $\frac{g_1(x)}{g_2(x)}$ is strictly decreasing on $(0,1)$.
\par
Moreover, making use of L'H\^{o}pital's rule leads to
\begin{equation}\label{eq2.7}
\lim_{x\to 0}\frac{g_1(x)}{g_2(x)}=\frac{1}{6p}
\end{equation}
and
\begin{equation}\label{eq2.8}
\lim_{x\to 1}\frac{g_1(x)}{g_2(x)}=\frac{\sqrt{2}\,t^*-1}{\sqrt{2}\,(2p-1)t^*+1}.
\end{equation}
\par
When $u\ge \frac{1}{6p}$, combining~\eqref{eq2.4} and~\eqref{eq2.7} with the monotonicity of $\frac{g_1(x)}{g_2(x)}$ shows that the function $f_{u,p}(x)$ is strictly increasing on $(0,1)$. Therefore, the positivity of $f_{u,p}(x)$ on $(0,1)$ follows from~\eqref{eq2.2} and the increasingly monotonicity of $f_{u,p}(x)$.
\par
When $u\le \frac{\sqrt{2}\,t^*-1}{\sqrt{2}\,(2p-1)t^*+1}$, combining~\eqref{eq2.4} and~\eqref{eq2.8} with the monotonicity of $\frac{g_1(x)}{g_2(x)}$ reveals that the function $f_{u,p}(x)$ is strictly decreasing on $(0,1)$. Hence, the negativity of $f_{u,p}(x)$ on $(0,1)$ follows from~\eqref{eq2.2} and the decreasingly monotonicity of $f_{u,p}(x)$.
\par
When $\frac{\sqrt{2}\,t^*-1}{\sqrt{2}\,(2p-1)t^*+1}<u<\frac{1}{6p}$, from~\eqref{eq2.4}, \eqref{eq2.7}, \eqref{eq2.8}, and the monotonicity of the ratio $\frac{g_1(x)}{g_2(x)}$, we conclude that there exists a number $x_{0}\in(0,1)$ such that $f_{u,p}(x)$ is strictly
decreasing in $(0,x_{0})$ and strictly increasing in $(x_{0},1)$.
Denote the limit in~\eqref{eq2.3} by $h_{p}(u)$. Then, from the above arguments, it follows that
\begin{equation}\label{eq2.10}
h_{p}\biggl(\frac{1}{6p}\biggr)=p\ln\biggl(1+\frac{1}{6p}\biggr)+\ln(t^*)>0
\end{equation}
and
\begin{equation}\label{eq2.11}
h_{p}\biggl(\frac{\sqrt{2}\,t^*-1}{\sqrt{2}\,(2p-1)t^*+1}\biggr) =p\ln\biggl[1+\frac{\sqrt{2}\,t^*-1}{\sqrt{2}\,(2p-1)t^*+1}\biggr]+\ln(t^*)<0.
\end{equation}
Since $h_{p}(u)$ is strictly increasing for $u>-1$, so it is also in $\bigl[\frac{\sqrt{2}\,t^*-1}{\sqrt{2}\,(2p-1)t^*+1},\frac1{6p}\bigr]$. Thus, the inequalities in~\eqref{eq2.10} and~\eqref{eq2.11} imply that the function $h_{p}(u)$ has a unique zero point $u_0=\bigl(\frac{1}{t^*}\bigr)^{1/p}-1\in \bigl(\frac{\sqrt{2}\,t^*-1}{\sqrt{2}\,(2p-1)t^*+1},\frac1{6p}\bigr)$ such that
$h_{p}(u)<0$ for $u\in \bigl[\frac{\sqrt{2}\,t^*-1}{\sqrt{2}\,(2p-1)t^*+1},u_0\bigr)$ and
$h_{p}(u)>0$ for $u\in \bigl(u_0,\frac1{6p}\bigr]$. As a result, combining~\eqref{eq2.2} and~\eqref{eq2.3} with the piecewise monotonicity of $f_{u,p}(x)$ reveals that $f_{u,p}(x)<0$ for all $x\in(0,1)$ if and only if $\frac{\sqrt{2}\,t^*-1}{\sqrt{2}\,(2p-1)t^*+1}<u<u_0$. The proof of Lemma~\ref{f-u-p-x-lem} is complete.
\end{proof}

\section{Proof of Theorem~\ref{th1.1}}

Now we are in a position to prove our Theorem~\ref{th1.1}.
\par
Since both $Q_{t,p}(a,b)$ and $M(a,b)$ are symmetric and homogeneous of degree $1$, without loss of generality, we assume that $a>b$. Let $x=\frac{a-b}{a+b}\in(0,1)$. From~\eqref{eq1.3} and~\eqref{eq1.8}, we obtain
\begin{equation*}
\begin{split}
\ln\frac{Q_{t,p}(a,b)}{T(a,b)}&=\ln\frac{Q_{t,p}(a,b)}{A(a,b)}-\ln\frac{T(a,b)}{A(a,b)}\\
&=p\ln\bigl[1+(1-2t)^2x^2\bigr]-\ln{x}+\ln{\arcsinh{x}}.
\end{split}
\end{equation*}
Thus, Theorem~\ref{th1.1} follows from Lemma~\ref{f-u-p-x-lem}.


\begin{thebibliography}{99}

\bibitem{anderson1}
G. D. Anderson, M. K. Vamanamurthy, and M. Vuorinen, \textit{Conformal Invariants, Inequalities, and Quasiconformal Maps}, John Wiley \& Sons, New York, 1997.

\bibitem{seiffert15}
Y.-M. Chu and S.-W. Hou, \emph{Sharp bounds for Seiffert mean in terms of contraharmonic mean}, Abstr. Appl. Anal. \textbf{2012} (2012), Article ID 425175, 6 pages; Available online at \url{http://dx.doi.org/10.1155/2012/425175}.

\bibitem{seiffert10}
Y.-M. Chu, S.-W. Hou, and Z.-H. Shen, \emph{Sharp bounds for Seiffert mean in terms of root mean square}, J. Inequal. Appl \textbf{2012}, 2012:11, 6~pages; Available online at \url{http://dx.doi.org/10.1186/1029-242X-2012-11}.

\bibitem{back3}
Y.-M. Chu, M.-K. Wang, and W.-M. Gong, \emph{Two sharp double inequalities for Seiffert mean}, J. Inequal. Appl. \textbf{2011}, 2011:44, 7 pages; Available online at \url{http://dx.doi.org/10.1186/1029-242X-2011-44}.

\bibitem{jiang-aadm4}
W.-D. Jiang, \emph{Some sharp inequalities involving Neuman-S\'andor mean and other means}, Appl. Anal. Discrete Math. (2012), in press.

\bibitem{Li-Long-Chu-JIA-12}
Y.-M. Li, B.-Y. Long, and Y.-M. Chu, \emph{Sharp bounds for the Neuman-S\'andor mean in terms of generalized logarithmic mean}, J. Math. Inequal. \textbf{6} (2012), no.~4, 567\nobreakdash--577; Available online at \url{http://dx.doi.org/10.7153/jmi-06-54}.

\bibitem{Neuman-jmi-06-62}
E. Neuman, \emph{A note on a certain bivariate mean}, J. Math. Inequal. \textbf{6} (2012), no.~4, 637\nobreakdash--643; Available online at \url{http://dx.doi.org/10.7153/jmi-06-62}.

\bibitem{Neuman-jmi-601-609}
E. Neuman, \emph{Inequalities for the Schwab-Borchardt mean and their applications}, J. Math. Inequal. \textbf{5} (2011), no.~4, 601\nobreakdash--609; Available online at \url{http://dx.doi.org/10.7153/jmi-05-52}.

\bibitem{Neuman-Sandor-Pannon}
E. Neuman and J. S\'andor, \emph{On the Schwab-Borchardt mean}, Math. Pannon. \textbf{14} (2003), no.~2, 253\nobreakdash--266.

\bibitem{Neuman-Sandor-Pannon2}
E. Neuman and J. S\'andor, \emph{On the Schwab-Borchardt mean II}, Math. Pannon. \textbf{17} (2006), no.~1, 49\nobreakdash--59.

\bibitem{background-Jiang-Qi-revised.tex}
W.-D. Jiang and F. Qi, \emph{Some sharp inequalities involving Seiffert and other means and their concise proofs}, Math. Inequal. Appl. \textbf{15} (2012), no.~4, 1007\nobreakdash--1017; Available online at \url{http://dx.doi.org/10.7153/mia-15-86}.

\bibitem{Jiang-Neuman.tex}
W.-D. Jiang and F. Qi, \emph{Sharp bounds for Neuman-S\'andor's mean in terms of the root-mean-square}, available online at \url{http://arxiv.org/abs/1301.3267}.

\bibitem{back10}
H.-J. Seiffert, \emph{Aufgabe $\beta$ 16}, Die Wurzel \textbf{29} (1995), 221\nobreakdash--222.

\end{thebibliography}
\end{document}